\documentclass[reqno, a4paper]{article}[12pt]

\makeatletter
\renewcommand*\l@section{\@dottedtocline{1}{1.5em}{2.3em}}
\makeatother

\usepackage{CJK}
\usepackage{amsmath}
 
\usepackage{amsfonts}
\usepackage{amssymb}
\usepackage{amsthm}
\usepackage{amssymb}
\usepackage{enumerate}
\usepackage[calc]{picture}
\usepackage[all,cmtip]{xy}

\usepackage[mathscr]{eucal}
\usepackage{eqlist}

\usepackage{abstract}

\usepackage[T1]{fontenc}

\theoremstyle{plain}
\newtheorem{theorem}{Theorem}
\newtheorem{proposition}[theorem]{Proposition}
\newtheorem{lemma}[theorem]{Lemma}
\newtheorem{corollary}[theorem]{Corollary}

\newtheorem{example}[theorem]{Example}
\newtheorem{definition}[theorem]{Definition}

\theoremstyle{remark}
\newtheorem{remark}[theorem]{Remark}

\numberwithin{theorem}{section}%{definition}%{remark}

\usepackage[top=1.8cm, bottom=1.8cm, left=2.5cm, right=2.5cm]{geometry}

\begin{document}

\title{Regular Maps on Cartesian Products and Disjoint Unions of Manifolds}

\author{Shiquan Ren}
			
\begin{center}
{\Large {\textbf{{Regular Maps on Cartesian Products and Disjoint Unions of Manifolds}}}}

\vspace{0.5cm}

{\large Shiquan Ren}

\end{center}

\begin{quote}

\begin{abstract}
\medskip

A map from a manifold to a Euclidean space is said to be $k$-regular if the image of any distinct $k$ points are linearly independent. For $k$-regular maps on manifolds, lower bounds of the dimension of the ambient Euclidean space have been extensively studied. 
In this paper, we study the lower bounds of the dimension of the ambient Euclidean space for $2$-regular maps on Cartesian products of manifolds. As corollaries, we obtain the exact lower bounds of the dimension of the ambient Euclidean space for $2$-regular maps and $3$-regular maps on spheres as well as on some real projective spaces. Moreover, generalizing the notion of $k$-regular maps, we study  the lower bounds of the dimension of the ambient Euclidean space for maps with certain non-degeneracy conditions from disjoint unions of manifolds into Euclidean spaces.  
\end{abstract}

\end{quote}
\vspace{0.2cm}

\noindent{\bf{AMS Mathematical Classifications 2010}}.  	 Primary 55R40; Secondary 55T10,  55R80, 53C40


%\tableofcontents

\section{Introduction}

Let $M$ be a smooth manifold and let $\mathbb{F}$ denote the real numbers  $\mathbb{R}$ or the complex numbers $\mathbb{C}$. 
For any $k\geq  2$, a map $f: M\longrightarrow \mathbb{F}^N$ is called (real or complex) $k$-regular 
if for any distinct $k$ points $x_1,\cdots, x_k$ in $M$,  $f(x_1), \cdots, f(x_k)$ are linearly independent in $\mathbb{F}^N$. For simplicity, a real $k$-regular map  is also called a $k$-regular map.  Any  (real or complex) $(k+1)$-regular map is (real or complex)  $k$-regular, and any $k$-regular map is injective.

Throughout this paper, all maps and functions are assumed to be continuous.  We use $S^m$ to denote  the $m$-sphere and use $\mathbb{R}P^m$, $\mathbb{C}P^m$ and $\mathbb{H}P^m$ to denote the real, complex and quaternionic projective spaces consisting of real lines through the origin in $\mathbb{R}^{m+1}$, complex lines through the origin in  $\mathbb{C}^{m+1}$  and  quaternionic lines through the origin in $\mathbb{H}^{m+1}$ respectively.  Moreover, we assume $m\geq 2$. 
\smallskip

In 1957, The study of $k$-regular maps was initiated by K. Borsuk \cite{Borsuk}. 
Later, the problem attracted additional attention because of its connection with the theory of Ceby${\check{\text{s}}}$ev approximation  
 (cf. \cite{handel4} and \cite[pp. 237-242]{singer}):
\begin{quote}
{\sc {Theorem}  [Haar-Kolmogorov-Rubinstein]. }Suppose $M$ is compact and $f_1,\cdots,f_n$ are linearly independent real-valued  functions on $M$. Let $F$ be the linear space spanned by $f_1,\cdots,f_n$ over $\mathbb{R}$. Then $(f_1,\cdots,f_n)$ is a $k$-regular map from $M$ to $\mathbb{R}^n$ if and only if for any  real-valued  function $g$ on $M$, the dimension of the set
$
\{f\in F\mid \sup_{x\in M}|g(x)-f(x)|=m\}
$
is smaller than or equal to $n-k$, where $m$ is the infimum of $\sup_{x\in M} |g(x)-f(x)|$ for all $f$ in $F$. 
\end{quote}

From 1970's to nowadays, $k$-regular maps on manifolds have been extensively studied. % in \cite{high1,high2,chi,cohen1,handel4,handel1,handel2,   handel3,  vas}.  
In 1978, some  $k$-regular maps on  the plane were constructed by F.R. Cohen and D. Handel  \cite{cohen1}:
\begin{quote}
{\sc \cite[Example~1.2]{cohen1}. }The map from $\mathbb{C}$ to $\mathbb{R}^{2k-1}$ sending $z$ to $(1,$ $z,$ $z^2,$ $\cdots,$ $z^{k-1})$ is $k$-regular.
\end{quote}
And in 2016, some $3$-regular maps on spheres were constructed by P. Blagojevi$\check{\text{c}}$,   W. L{\"u}ck and G. Ziegler \cite{high1}:
\begin{quote}
{\sc \cite[Example~2.6-(2)]{high1}. }Let $i$ be the standard embedding from $S^m$  to $\mathbb{R}^{m+1}$  and $1$ the constant map with image $1$. Then the map $(1,i)$ from $S^m$  to $\mathbb{R}^{m+2}$ is  $3$-regular. 
\end{quote}    
On the other hand,  generalizing the results of M.E. Chisholm \cite{chi} in 1979 and F.R. Cohen and D. Handel \cite{cohen1} in 1978, the lower bounds of $N$ for $k$-regular maps of Euclidean spaces into $\mathbb{R}^N$ were studied by P. Blagojevi$\check{\text{c}}$,   W. L{\"u}ck and G. Ziegler \cite{high1} in 2016:
\begin{quote}
{\sc \cite[Theorem 2.1]{high1}. }Let $\alpha(k)$ denote the number of ones in the dyadic expansion of $k$.  If there exists a $k$-regular map of $\mathbb{R}^m$ into $\mathbb{R}^{N}$, then $N\geq m(k-\alpha(k))+\alpha(k)$.
\label{pavle2014}
\end{quote}
While  the lower bounds of $N$ for complex $k$-regular maps of Euclidean spaces into $\mathbb{C}^N$ were studied by  P. Blagojevi$\check{\text{c}}$,  F.R. Cohen, W. L{\"u}ck and G. Ziegler  \cite{high2} in 2015:
\begin{quote}
{\sc \cite[Theorem 5.2]{high2}. }Let $p$ be an odd prime. If there exists a complex $p$-regular map of $\mathbb{R}^m$ into $\mathbb{C}^{N}$, then $N\geq [\frac{m+1}{2}](p-1)+1$.

{\sc \cite[Theorem 5.3]{high2}. }Let $p$ be an odd prime and let $\alpha_p(k)$ be the sum of coefficients in the $p$-adic expansion of $k$.   If $m$ is a power of $p$ and there exists a complex $k$-regular map of $\mathbb{C}^m$ into $\mathbb{C}^N$, then $N\geq m(k-\alpha_p(k))+\alpha_p(k)$. 
\end{quote}

Compared with $k$-regular maps on Euclidean spaces, it is more difficult to study the dimension of the ambient Euclidean space for $k$-regular maps on general manifolds. In 1996, D. Handel \cite{1996}  considered the lower bounds of $N$ for $2k$-regular maps of  manifolds into $\mathbb{R}^N$. In 2011,  R. Karasev \cite{karasev} pointed out a gap of the proof \cite[p. 1611]{1996}.   %In this paper, as a by-product of Theorem~1.1, we will give a  counter-example of \cite[Theorem~4.4]{karasev} in Corollary~\ref{sphere}.
  %However, the problem has not been completely solved in \cite{karasev}. For instance,  a  counter-example for \cite[Theorem~4.4]{karasev} is Corollary~\ref{sphere}. 
 
 \smallskip
 
 The family of $k$-regular maps on disjoint unions of manifolds is a particular family of $k$-regular maps on general manifolds. Problems concerning $k$-regular maps on disjoint unions of manifolds attracted attention since 1980's and was firstly considered  by D. Handel \cite{handel2}:
 \begin{quote} 
{\sc \cite[Theorem 2.4]{handel2}. }
\label{c2th2}
Let $M_1,\cdots,M_k$ be closed, connected manifolds of dimensions $n_1,\cdots,n_k$ respectively. Suppose for $1\leq i\leq k$, the $q_i$-th dual Stiefel-Whitney class of $M_i$ is non-zero. If there exists a $2k$-regular map of the disjoint union $\coprod _{i=1}^kM_i$ into $\mathbb{R}^N$, then $N\geq 2k+\sum_{i=1}^k(n_i+q_i)$.
\end{quote}
 \smallskip

 In this paper, our first aim is to give a lower bound of $N$ for $2$-regular maps on Cartesian products of spheres and real, complex and quaternionic projective spaces into $\mathbb{R}^N$.  We prove the next theorem.
 
 \begin{theorem}[Main Theorem I]  \label{0921-3}
Suppose there is a $2$-regular map
\begin{eqnarray*}
f: \prod_{i=1}^{k_1}S^{m_{1,i}}\times \prod_{j=1}^{k_2}\mathbb{R}P^{m_{2,j}} \times \prod_{t=1}^{k_3}\mathbb{C}P^{m_{3,t}}\times \prod_{l=1}^{k_4}\mathbb{H}P^{m_{4,l}}\longrightarrow \mathbb{R}^N.
\end{eqnarray*}
Then 
\begin{eqnarray*}
N&\geq& \sum_{i=1}^{k_1} m_{1,i}+\sum_{j=1}^{k_2}2^{[\log_2 m_{2,j}]+1}+\sum_{t=1}^{k_3}2^{[\log_2 m_{3,t}]+2} \\&&+\sum_{l=1}^{k_4}2^{[\log_2 m_{4,l}]+3}-k_2-2k_3-4k_4 +2.
\end{eqnarray*}
\end{theorem}

The following corollaries follow from Theorem~\ref{0921-3}. 
\begin{corollary}
\label{rpm}
Let $2^i\leq m<2^{i+1}$, $i\geq 1$. 

 (a). If there exists a $2$-regular map of $\mathbb{R}P^m$ into $\mathbb{R}^N$, then $N\geq 2^{i+1}+1$. 

(b). If there exists a $2$-regular map of $\mathbb{C}P^m$ into $\mathbb{R}^N$, then $N\geq 2^{i+2}$. 

(c). If there exists a $2$-regular map of $\mathbb{H}P^m$ into $\mathbb{R}^N$, then $N\geq 2^{i+3}-2$. 
\end{corollary}

\begin{corollary}\label{sphere}
The following are equivalent

(a).  there exists a $3$-regular map of $S^m$ into $\mathbb{R}^N$,

(b). there exists a $2$-regular map of $S^m$ into $\mathbb{R}^N$,

(c). $N\geq m+2$. 
\end{corollary}

\begin{remark}
Corollary~\ref{sphere} gives a  counter-example of \cite[Theorem~4.4]{karasev}.
\end{remark}

\begin{corollary}\label{exactrpm}
Let $m=2^i+1$, $i\geq 1$. Then the following are equivalent

(a). there exists a $3$-regular map of $\mathbb{R}P^m$ into $\mathbb{R}^N$,

(b). there exists a $2$-regular map of $\mathbb{R}P^m$ into $\mathbb{R}^N$,

(c). $N\geq 2m-1$.
\end{corollary} 
  
 \smallskip
 
Besides considering $k$-regular maps, we generalize the notion of $k$-regular maps on disjoint unions of manifolds  and   consider a weaker non-degeneracy condition for maps on disjoint unions of manifolds.  We give  the  following definition. 
\begin{definition}
\label{c2def1}
Let $M_1,$ $\cdots,$  $M_n$ be manifolds.  A map 
\begin{eqnarray*}\label{c2eq1}
f: \coprod_{i=1}^n M_i\longrightarrow \mathbb{F}^N 
\end{eqnarray*}
is called (real or complex) $(M_1,k_1;M_2,k_2;\cdots;M_n,k_n)$-regular 
if for any distinct points
$
x_{i,1},$ $x_{i,2},$ $\cdots,$ $x_{i,k_i}\in M_i,
$
$i=1,2,$ $\cdots,$ $n$, their images 
\begin{eqnarray*}
\coprod_{ i=1} ^ n\{ f(x_{i,1}), f(x_{i,2}), \cdots, f(x_{i,k_i})\}
\end{eqnarray*}
are linearly independent in $\mathbb{F}^N$. 
In particular, a real $(M_1,k_1;$ $M_2,k_2;$ $\cdots;$ $M_n,k_n)$-regular map  is called $(M_1,k_1;$ $M_2,k_2;$ $\cdots;$ $M_n,k_n)$-regular for short.
\end{definition}

 With the help of Definition~\ref{c2def1}, we re-obtain \cite[Theorem 2.4]{handel2} in Remark~\ref{1007-r2}.  
 
 \smallskip

The second aim of this paper is to give a lower bound  of $N$ for the regular maps  defined in Definition~\ref{c2def1}  from disjoint unions of planes, spheres and real, complex and quaternionic projective spaces into $\mathbb{R}^N$. We prove the next theorem.  

\begin{theorem}[Main Theorem II] \label{c2th1}
Suppose we have a $(\mathbb{R}^2,2^{d_1};$ $\cdots;$ $\mathbb{R}^2,2^{d_{k_0}};$ $S^{m_{1,1}},2;$ $\cdots;$ $S^{m_{1,k_1}},2;$ $ \mathbb{R}P^{m_{2,1}},2;$ $\cdots;$  $\mathbb{R}P^{m_{2,k_2}},2;$ 
$\mathbb{C}P^{m_{3,1}},2;$ $\cdots;$ $\mathbb{C}P^{m_{3,k_3}},2;$ $\mathbb{H}P^{m_{4,1}},2;$ $\cdots;$ $\mathbb{H}P^{m_{4,k_4}},2)$-regular map
\begin{eqnarray*}
f: (\coprod_{k_0} \mathbb{R}^2)\coprod (\coprod_{i=1}^{k_1} S^{m_{1,i}})\coprod (\coprod_{j=1}^{k_2}\mathbb{R}P^{m_{2,j}})\coprod(\coprod_{t=1}^{k_3}\mathbb{C}P^{m_{3,t}})\coprod(\coprod_{l=1}^{k_4} \mathbb{H}P^{m_{4,l}})\longrightarrow \mathbb{R}^N.
\end{eqnarray*}   
Then
\begin{eqnarray*} 
N&\geq& \sum_{s=1}^{k_0} 2^{d_s+1}+\sum_{i=1}^{k_1}m_{1,i}+\sum_{j=1}^{k_2}2^{[\log_2m_{2,j}]+1}\nonumber\\
&&+\sum_{t=1}^{k_3} 2^{[\log_2 m_{3,t}]+2}
+\sum_{l=1}^{k_4}2^{[\log_2 m_{4,l}]+3}-k_0+ 2k_1+k_2-2k_4.
\end{eqnarray*}
\end{theorem}

The next corollary follows from Theorem~\ref{c2th1}. 
\begin{corollary}\label{c2th3}
The following are equivalent:

(a). there exists a $(\mathbb{R}^2, 2^{d_1};\cdots; \mathbb{R}^2, 2^{d_{k_0}};$ $S^{m_1},2;\cdots; S^{m_{k_1}},2;$ $\mathbb{R}P^{2^{u_1}+1},2;$ $\cdots;$ $\mathbb{R}P^{2^{u_{k_2}}+1},2)$-regular map
\begin{eqnarray*}
f: (\coprod_{k_0}\mathbb{R}^2)\coprod(\coprod_{i=1}^{k_1} S^{m_i})\coprod(\coprod _{j=1}^{k_2}\mathbb{R}P^{2^{u_j}+1})\longrightarrow \mathbb{R}^N;
\end{eqnarray*}
 
 (b). there exists a $(\mathbb{R}^2, 2^{d_1};\cdots; \mathbb{R}^2, 2^{d_{k_0}};$ $S^{m_1},3;\cdots; S^{m_{k_1}},3;$ $\mathbb{R}P^{2^{u_1}+1},3;$ $\cdots;$ $\mathbb{R}P^{2^{u_{k_2}}+1},3)$-regular map
\begin{eqnarray*}
f: (\coprod_{k_0}\mathbb{R}^2)\coprod(\coprod_{i=1}^{k_1} S^{m_i})\coprod(\coprod _{j=1}^{k_2}\mathbb{R}P^{2^{u_j}+1})\longrightarrow \mathbb{R}^N;
\end{eqnarray*}

(c). $$
N\geq \sum_{s=1}^{k_0} 2^{d_s+1}+\sum_{i=1}^{k_1} m_i+ \sum_{j=1}^{k_2} 2^{u_j+1}-k_0+2k_1+k_2.
$$
\end{corollary}

\smallskip

As by-products, we  give some lower bounds of $N$ for the complex regular maps   from Euclidean spaces, spheres and complex projective spaces into $\mathbb{C}^N$.  In the last section of this paper, we give a generalization of  \cite[Theorem~5.2]{high2}:
\begin{itemize}
\item
For an    odd prime $p$, if there exists a complex $np$-regular map from $ \mathbb{R}^m$ into  $\mathbb{C}^N$, then $N\geq n([\frac{m+1}{2}] (p-1)+1)$. 
\end{itemize}
Moreover, we prove the following:
\begin{itemize}
\item
If there exists a  complex $2$-regular map of $S^m$ into $\mathbb{C}^N$, then $N\geq [\frac{m}{2}] +2$.  
\item
For $m\geq 4$, if there exists a complex $2$-regular map of $\mathbb{C}P^m$ into $\mathbb{C}^N$, then $N\geq 2m$. 
\end{itemize}

\smallskip

The paper is organized as follows. In Section~\ref{subs2.1}, we give  some examples of $k$-regular maps as well as the regular maps on disjoint unions of manifolds defined in Definition~\ref{c2def1}. In Section~\ref{subs2.2}, we review the cohomology of Grassmannians.  In Section~\ref{subs2.4}, we firstly review the Stiefel-Whitney classes and Chern classes of the canonical vector bundle over configuration spaces. Then we prove some auxiliary lemmas. In Section~\ref{sec5}, we prove Theorem~\ref{0921-3}.  In Section~\ref{6.2},  we give an obstruction for the regular maps on disjoint unions of manifolds defined in Definition~\ref{c2def1}. In Section~\ref{6.3}, we prove Theorem~\ref{c2th1}. In Section~\ref{sec8}, we prove  the lower bounds of $N$ (listed above) for complex regular maps from Euclidean spaces, spheres and complex projective spaces into $\mathbb{C}^N$.

The main results of this paper are Theorem~\ref{0921-3} and Theorem~\ref{c2th1}.

\section{Examples of regular maps}\label{subs2.1}

With the help of \cite[Example~2.6-(2)]{high1}, if there exists an embedding of  $M$ into $\mathbb{R}^n$, then composed with an embedding of $\mathbb{R}^n$ into $S^n$ and a $3$-regular map of $S^n$ into $\mathbb{R}^{n+2}$, we obtain a $3$-regular map of $M$ into $\mathbb{R}^{n+2}$. Let $M$ be $\mathbb{R}P^m$. By applying \cite[Theorem~4.1]{handel5}, \cite[Theorem~5.2, Theorem~5.7]{mark} and \cite[Theorem~5]{rees}, we have the following example.

\begin{example}\label{upperrpm}
There exist $3$-regular maps of $\mathbb{R}P^m$ into $\mathbb{R}^N$ for the cases listed in Table~\ref{table:table6.1}.
\end{example}

\begin{table}[h]
\caption{$3$-regular maps of $\mathbb{R}P^m$ into $\mathbb{R}^N$ } % title of Table
\vspace{0.1cm}

\label{table:table6.1}

\centering % used for centering table
\begin{tabular}{  | c | c | c |        } % centered columns (4 columns)
%inserts double horizontal lines
 % \hline 
%	&$m$ & $N$ \\   % inserts table
\hline % inserts single horizontal line
 %$\mathbb{R}P^m$ 
 $m=8q+3$ or $8q+5$, $q>0$ & $N\geq 2m-\min\{5,\alpha(q)\}$\\
               $m=8q+1$, $q>0$           & $N\geq 2m-\min\{7, \alpha(q)\}+2$\\
               $m=32q+7$, $q>0$          & $N\geq 2m-6$\\
               $m=8q+7$, $q>1$           & $N\geq 2m-5$\\
               $m\equiv 3$ (mod $8$), $m\geq 19$    & $N\geq 2m-4$\\
               $m\equiv 1$ (mod $4$), $m\neq 2^i+1$ & $N\geq 2m-2$\\
               $m=4q+i$, $i=0$ or $2$,  $q\neq 2^j$ or $0$ & $N\geq 2m-1$\\
               $m=2^j+1$, $j\geq 2$                        &  $N\geq 2m-1$\\
               $m=2^j+2$, $j\geq 3$                        & $N\geq 2m$ \\
\hline %inserts single line
\end{tabular}
\label{x} % is used to refer this table in the text
\end{table}

Let $f_i: M_i\longrightarrow \mathbb{F}^{N_i}$  be a (real or complex) $k_i$-regular map for $i=1,2,\cdots, n$.  Then we have a (real or complex) $(M_1,k_1;M_2,k_2;\cdots;M_n,k_n)$-regular map 
$$
f: \coprod_{i=1}^n M_i\longrightarrow  \prod _{i=0}^n \mathbb{F}^{N_i} \cong\mathbb{F}^{\sum_{i=1}^n N_i}
$$ 
given by $f(x)=(0,\cdots,0, f_i(x), 0,\cdots,0)$ for $x\in M_i$, $i=1,2,\cdots,n$. 
Consequently, with the helps of  \cite[Example~1.2]{cohen1}, \cite[Example~2.6-(2)]{high1} and Example~\ref{upperrpm},  we have the following example.
\begin{example}\label{c2pr2}
There exists a  $(\mathbb{R}^2,b_1;$ $\cdots;$ $\mathbb{R}^2,b_{k_0};$  $S^{m_{1}},3;$ $\cdots;$ $S^{m_{k_1}},3;$ $\mathbb{R}P^{2^{u_1}+1},3;$ $\cdots;$ $\mathbb{R}P^{2^{u_{k_2}}+1},3)$-regular map
\begin{eqnarray*}
f: (\coprod_{k_0} \mathbb{R}^2)\coprod (\coprod_{i=1}^{k_1}S^{m_i})\coprod(\coprod_{j=1}^{k_2}\mathbb{R}P^{2^{u_j}+1})\longrightarrow \mathbb{R}^N
\end{eqnarray*}
for any
$$
N\geq{2(\sum_{s=1}^{k_0}b_s)+\sum_{i=1}^{k_1} m_i + \sum_{j=1}^{k_2} 2^{u_j+1}-k_0+2k_1+k_2}.
$$
\end{example}

\section{Cohomology of Grassmannians}\label{subs2.2}

Let $A$ be a ring and $a\in A\setminus A^*$ where $A^*$ is the set of  invertible elements of $A$. The height of  $a$ is defined as the smallest positive integer $n$ such that $a^{n+1}=0$, or infinity, if such an $n$ does not exist (cf. \cite{high2}).

Given positive integers $k$ and $n$ with $n\geq k$, let $G_k(\mathbb{\mathbb{F}}^{n+1})$ be the (real or complex) Grassmannian consisting of $k$-dimensional subspaces of $\mathbb{F}^{n+1}$ and $G_k(\mathbb{\mathbb{F}}^{\infty})$ the direct limit of $G_k(\mathbb{\mathbb{F}}^{n+1})$.  The canonical inclusion of $\mathbb{F}^{n+1}$ into  $\mathbb{F}^\infty$ as the first $(n+1)$-coordinates induces an inclusion 
\begin{eqnarray}
\label{inclusion}
i: G_k(\mathbb{F}^{n+1})\longrightarrow G_k(\mathbb{F}^\infty).
\end{eqnarray}

\noindent{\sc{Case~1}: $\mathbb{F}=\mathbb{R}$. }

\noindent It is known that
$H^*(G_k(\mathbb{R}^\infty);\mathbb{Z}_2)=\mathbb{Z}_2[w_1,w_2,\cdots,w_k]$
where $w_i$ is the $i$-th universal Stiefel-Whitney class with $|w_i|=i$.
And
\begin{eqnarray}\label{0909-1}
H^*(G_k(\mathbb{R}^{n+1});\mathbb{Z}_2)=\mathbb{Z}_2[w_1,w_2,\cdots,w_k]/(\bar w_{n-k+2},\bar w_{n-k+3},\cdots,\bar w_{n+1})
\end{eqnarray}
where $\bar w_j$ is defined as the $j$-th degree term in the expansion of $(1+w_1+\cdots+w_k)^{-1}$ and $(\bar w_{n-k+2},\bar w_{n-k+3},\cdots,\bar w_{n+1})$ is the ideal generated by $\bar w_{n-k+2}$, $\bar w_{n-k+3}$, $\cdots$, $\bar w_{n+1}$. The inclusion (\ref{inclusion}) induces an epimorphism on mod 2 cohomology.

\smallskip

\noindent{\sc{Case~2}: $\mathbb{F}=\mathbb{C}$. }

\noindent
It is known that
$H^*(G_k(\mathbb{C}^\infty);\mathbb{Z})=\mathbb{Z}[c_1,c_2,\cdots,c_k]$
where $c_i$ is the $i$-th universal Chern class with $|c_i|=2i$. 
And
\begin{eqnarray}\label{0909-2}
H^*(G_k(\mathbb{C}^{n+1});\mathbb{Z})=\mathbb{Z}[c_1,c_2,\cdots,c_k]/(\bar c_{n-k+2},\bar c_{n-k+3},\cdots,\bar c_{n+1})
\end{eqnarray}
where $\bar c_j$ is defined as the $2j$-th degree term in the expansion of $(1+c_1+\cdots+c_k)^{-1}$ and $(\bar c_{n-k+2},\bar c_{n-k+3},\cdots,\bar c_{n+1})$ is the ideal generated by $\bar c_{n-k+2}$, $\bar c_{n-k+3}$, $\cdots$, $\bar c_{n+1}$. The inclusion  (\ref{inclusion}) induces an epimorphism  on integral cohomology.

\begin{proposition}
\label{schubert2}
Let $n\geq k\geq 1$. Then in (\ref{0909-2}),
\begin{eqnarray*}
\text{height}(c_1)=\text{dim}_{\mathbb{C}}G_k(\mathbb{C}^{n+1}).
\end{eqnarray*}
In particular, if $h(n)$ denotes the height of $c_1$ in $H^*(G_2(\mathbb{C}^{n+1});\mathbb{Z})$, then $h(n)=2n-2$. 
\end{proposition}
\begin{proof}
The Pluecker embedding embeds   $G_k(\mathbb C^{n+1})$  as a subvariety into  $\mathbb{C}P^r$, where $r$ is the complex dimension of the $k$-th exterior power of $\mathbb{C}^{n+1}$.  It follows  that $c_1$ is the pull-back a generator of $H^2(\mathbb {C}P^r;\mathbb{Z})$, which is an ample class. Since the complex dimension of $G_k(\mathbb{C}^{n+1})$ is  $k(n+1-k)$, it follows that  $c_1^{k(n+1-k)}$ is non-zero and $c_1^{k(n+1-k)+1}$ is zero.  The assertion follows.  
\end{proof}
\begin{remark}
For the backgrounds of the above proof, we  may refer to \cite{griff}. 
\end{remark}

\section{Characteristic classes of the canonical vector bundle over configuration spaces}\label{subs2.4}

Let $\Sigma_k$ be the permutation group of order $k$ and let the $k$-th configuration space of $M$ be
\begin{eqnarray*}
F(M,k)=\{(x_1,\cdots, x_k)\in M\times \cdots \times M\mid \text{for any }i\neq j, x_i\neq x_j\}.
\end{eqnarray*}
For any $\sigma\in \Sigma_k$, let $\sigma$ act on $F(M,k)$ by
\begin{eqnarray*}
(x_1,\cdots,x_k)\sigma=(x_{\sigma(1)},\cdots,x_{\sigma(k)}) 
\end{eqnarray*}
and act on $\mathbb{F}^k$ by
\begin{eqnarray*}
\sigma(r_1,\cdots,r_k)=(r_{\sigma^{-1}(1)},\cdots,r_{\sigma^{-1}(k)}).
\end{eqnarray*}
Then we have a space  $F(M,k)/\Sigma_k$, called the $k$-th unordered configuration space of $M$,  and an
 $O(\mathbb{F}^k)$-bundle
\begin{eqnarray*}
\xi_{M,k}^{\mathbb{F}}: \mathbb{F}^k\longrightarrow F(M,k)\times_{\Sigma_k}\mathbb{F}^k\longrightarrow F(M,k)/\Sigma_k.
\end{eqnarray*}
We denote the classifying map of $\xi_{M,k}^\mathbb{F}$ as $$h:F(M,k)/\Sigma_k\longrightarrow G_k(\mathbb{F}^\infty).$$ 
We consider a $\Sigma_k$-invariant subspace  $W_{k}^\mathbb{F}$ of $\mathbb{F}^k$ consisting of vectors $(x_1,x_2,\cdots,x_k)$ such that $\sum_{i=1}^k x_i=0$.  Then we have a $O(\mathbb{F}^{k-1})$-bundle 
\begin{eqnarray*}
\zeta_{M,k}^\mathbb{F}: W_{k}^\mathbb{F}\longrightarrow F(M,k)\times_{\Sigma_k}W_{k}^\mathbb{F}\longrightarrow F(M,k)/\Sigma_k.
\end{eqnarray*}
Let $\epsilon_{M,k}^{\mathbb{F}}$ be the trivial $\mathbb{F}^1$-bundle over $F(M,k)/\Sigma_k$. Then we have an isomorphism of vector bundles
\begin{eqnarray}\label{0930-6}
\xi_{M,k}^\mathbb{F}\cong \zeta_{M,k}^\mathbb{F}\oplus \epsilon_{M,k}^\mathbb{F}.
\end{eqnarray}
For simplicity,    we omit the symbol $\mathbb{F}$ in (\ref{0930-6}) if $\mathbb{F}=\mathbb{R}$. 

\smallskip

For a vector bundle $\eta$, we denote $w_i(\eta)$ and $w(\eta)$ as its $i$-th Stiefel-Whitney class and its total Stiefel-Whitney class respectively, and  $\bar w(\eta)=1/w(\eta)$ the dual Stiefel-Whitney class. For a complex vector bundle $\eta^\mathbb{C}$, we denote $c_i(\eta^\mathbb{C})$ and $c(\eta^\mathbb{C})$ as its $i$-th Chern class and its total Chern class respectively, and  $\bar c(\eta^{\mathbb{C}})=1/c(\eta^{\mathbb{C}})$ the dual Chern class. It follows from (\ref{0930-6}) that
\begin{eqnarray}
\label{0703r}
w_k(\xi_{M,k})&=&0,\\
\label{0703c}
c_k(\xi_{M,k}^\mathbb{C})&=&0.
\end{eqnarray} 
By detecting the dual Stiefel-Whitney class and the dual Chern class, the next two lemmas give    obstructions for (real and complex) $k$-regular maps. 

\begin{lemma}\cite{cohen1}  \label{cor1}
Lex $f: M\longrightarrow \mathbb{R}^N$ be a $k$-regular map on the manifold $M$. If $\bar w_t(\xi_{M,k})\neq 0$, then $N\geq t+k$.
\end{lemma}

\begin{lemma}\cite{high2} \label{cor2}
Let  $f: M\longrightarrow \mathbb{C}^N$ be a complex $k$-regular map on the manifold $M$. If $\bar c_t(\xi_{M,k}^\mathbb{C})\neq 0$, then $N\geq t+k$.
\end{lemma}

The next lemma gives the height of the first Stiefel-Whitney class of a closed, connected manifold. 

\begin{lemma}\cite{handel2,wu}\label{lewu}
Let $M$ be a closed, connected $m$-dimensional manifold. Suppose $q$ is the largest integer such that $\bar w_q(M)\neq 0$. Then 
\begin{eqnarray*}
w_1(\xi_{M,2})^{m+q}&\neq& 0,\\
w_1(\xi_{M,2})^{m+q+1}&=&0.
\end{eqnarray*}
 \end{lemma}
 
The next corollary follows from Lemma~\ref{lewu}.  
 
 \begin{corollary}\label{c2co3}
 Let $M$ be a closed, connected $m$-dimensional manifold. Suppose $q$ is the largest integer such that $\bar w_q(M)\neq 0$. Then 
\begin{eqnarray*}
\bar w_{m+q}(\xi_{M,2})&\neq& 0,\\
\bar w_{m+q+1}(\xi_{M,2})&=&0.
\end{eqnarray*} 
\end{corollary}
\begin{proof}
With the help of Lemma~\ref{lewu},
\begin{eqnarray*}
\bar w(\xi_{M,2})&=&(1+w_1(\xi_{M,2}))^{-1}\\
&=&1+\sum_{l=1}^{m+q}(-1)^lw_1(\xi_{M,2})^l.
\end{eqnarray*}
This implies Corollary~\ref{c2co3}.
\end{proof}

For a connected manifold $M$, we notice that  $F(M,k)$ is connected and the covering map from $F(M,k)$ to $F(M,k)/\Sigma_k$ induces an epimorphism 
\begin{eqnarray*}
\pi: \pi_1(F(M,k)/\Sigma_k)\longrightarrow \Sigma_k. 
\end{eqnarray*}
Let $r: \Sigma_k\longrightarrow  \{\pm 1\}$ be the sign representation of $\Sigma_k$.    Since  $\pi$ is surjective and 
$ r$ is non-trivial,  the map $r\circ \pi$ is non-trivial.   Moreover, it is direct to verify that  there is a bijection between $\text{Hom}(\pi_1(F(M,k)/\Sigma_k),\mathbb{Z}_2)$  and $\text{Vect}_{\mathbb{R}}^1(F(M,k)/\Sigma_k)$, and this bijection sends $r\circ \pi$ to the determinant line bundle of $\xi_{M,k}$.   Consequently, the determinant line bundle of $\xi_{M,k}$ is non-trivial.   Therefore, $\xi_{M,k}$ is non-orientable.  The next lemma follows.  

\begin{lemma}\label{0309-l1}
Let $M$ be a connected manifold.  Then  $w_1(\xi_{M,k})\neq 0$.
\end{lemma}

The next two theorems give the cohomology ring of the second unordered configuration space of complex projective spaces. 

\begin{theorem}\cite[Theorem 4.9]{yasui} \label{thm0419}
As a $H^*(G_2(\mathbb{C}^{m+1});\mathbb{Z}_2)$-module, the cohomology 
\begin{eqnarray}\label{con1}
H^*(F(\mathbb{C}P^m,2)/\Sigma_2;\mathbb{Z}_2)
\end{eqnarray}
 has $\{1,v,v^2\}$ as a basis. 
Moreover, the ring structure of (\ref{con1}) is given by $v^3=e_1 v$. Here $v=w_1(\xi_{\mathbb{C}P^m,2})$.
\end{theorem}
\begin{theorem}\cite[Theorem 4.10]{yasui} \label{thm0503}
As a $H^*(G_2(\mathbb{C}^{m+1});\mathbb{Z})$-module, the cohomology 
\begin{eqnarray}\label{con2}
H^*(F(\mathbb{C}P^m,2)/\Sigma_2;\mathbb{Z})
\end{eqnarray}
 has $\{1,u\}$ as generators with $|u|=2$. Moreover, the ring structure of (\ref{con2}) satisfies
$2u=0$ and $u^2=c_1u$. Here $c_1=c_1(\xi_{\mathbb{C}P^m,2}^{\mathbb{C}})$. 
\end{theorem}

In the remaining part of this section, we prove some auxiliary lemmas. 

\begin{lemma}\label{c2pr10}
For $i\geq 1$, the largest integer $\lambda$ such that $\bar w_\lambda(\xi_{\mathbb{R}^2,2^i})\neq 0$ is $\lambda=2^i-1$.
\end{lemma}
\begin{proof}
By  \cite[Theorem~3.1]{cohen1},  we see that $ w_{2^i-1}(\xi_{\mathbb{R}^2,2^i})\neq 0$. It is also given in \cite{cohen1} that 
$\bar w(\xi_{\mathbb{R}^2,2^i})=w(\xi_{\mathbb{R}^2,2^i})$. Thus we have $\bar w_{2^i-1}(\xi_{\mathbb{R}^2,2^i})\neq 0$. On the other hand, by   \cite[ Example~1.2]{cohen1}, we see that $\bar w_{2^i}(\xi_{\mathbb{R}^2,2^i})= 0$. Consequently,  $\lambda=2^i-1$.
\end{proof}
It follows from a geometric observation that $F(S^m,2)/\Sigma_2$ is homotopy equivalent to $\mathbb{R}P^m$.  Consequently,
\begin{eqnarray}
\label{cohomology2}
H^*(F(S^m,2)/\Sigma_2;\mathbb{Z}_2)&=&\mathbb{Z}_2[u]/(u^{m+1}), \text{\ \ \ \ \  \ \ \  }|u|=1; \\  
\label{cohomology1}
H^*(F(S^m,2)/\Sigma_2;\mathbb{Z})&=&
\left\{
\begin{array}{lcl}
   \mathbb{Z}[x]/(2x, x^{\frac{m+2}{2}}), \text{\   }|x|=2, \text{\ if \ }m \text{\ is even},  \\
\mathbb{Z}[x]/(2x, x^{\frac{m+1}{2}}), \text{\   }|x|=2,  \text{\  if \ }m \text{\ is odd}.
\end{array}
\right.
\end{eqnarray}\begin{lemma}\label{c2pr9}
(a). The largest integer $q$ such that $\bar w_q(S^m)\neq 0$ is $q=0$;
(b). the largest integer $\lambda$ such that $\bar w_\lambda(\xi_{S^m,2})\neq 0$ is $\lambda=m$.
\end{lemma}
\begin{proof}
By \cite[Corollary 11.15]{cc} and (\ref{cohomology2}), we obtain $w(S^m)=1$. With the help of Corollary~\ref{c2co3},  we have $\lambda=m$. 
\end{proof}

\begin{lemma}\label{0922-le3}
The largest integer $\tau$ such that $\bar c_\tau(\xi_{S^m,2}^\mathbb{C})\neq 0$ is 
$
\tau=[\frac{m}{2}]
$. 
\end{lemma}

\begin{proof}
Let $(\xi_{S^m,2}^\mathbb{C})_\mathbb{R}$ denote the underlying real vector bundle of $\xi_{S^m,2}^\mathbb{C}$. Then
\begin{eqnarray}\label{0703-5}
(\xi_{S^m,2}^\mathbb{C})_\mathbb{R}\cong(\xi_{S^m,2})^{\oplus 2}.
\end{eqnarray}
Moreover,  
since $u$ is the unique nonzero element in $H^1(F(S^m,2)/\Sigma_2;\mathbb{Z}_2)$,  it follows with the help of Lemma~\ref{0309-l1} that 
\begin{eqnarray}\label{0703-1}
w_1(\xi_{S^m,2})=u.
\end{eqnarray}
It follows from (\ref{0703r}), (\ref{0703-5}) and (\ref{0703-1})  that
\begin{eqnarray}\label{0703-6}
w((\xi_{S^m,2}^\mathbb{C})_\mathbb{R})&=&(w(\xi_{S^m,2}))^2\nonumber \\
&=&(1+u)^2\nonumber\\
&=&1+u^2.
\end{eqnarray}
Since $m\geq 2$, by (\ref{cohomology2}), $u^2\neq 0$. Consequently, (\ref{0703-6}) implies $w((\xi_{S^m,2}^\mathbb{C})_\mathbb{R})\neq 1$.
Hence the underlying real vector bundle of $\xi_{S^m,2}^\mathbb{C}$ is not trivial and it follows that  $\xi_{S^m,2}^\mathbb{C}$ is not trivial as a complex vector bundle. With the help of (\ref{0930-6}), we see that as a complex line bundle, $\zeta_{S^m,2}^\mathbb{C}$ is not trivial. Since the triviality of a complex line bundle is equivalent to the vanishness of its first Chern class,  we have
\begin{eqnarray}\label{0703-11}
c_1(\xi_{S^m,2}^\mathbb{C})&=&c_1(\zeta_{S^m,2}^\mathbb{C})\nonumber\\
&\neq& 0.
\end{eqnarray}
Since $x$ is the unique nontrivial element in $H^2(F(S^m,2)/\Sigma_2;\mathbb{Z})$, it follows from (\ref{0703-11}) that
\begin{eqnarray}\label{0703-7}
c_1(\xi_{S^m,2}^\mathbb{C})=x.
\end{eqnarray}
From (\ref{0703c}) and (\ref{0703-7}) we obtain
\begin{eqnarray}\label{0703-9}
\bar c_t( \xi_{S^m,2}^\mathbb{C})=(-1)^t x^t.
\end{eqnarray}
By applying (\ref{cohomology1}) to (\ref{0703-9}), we finish the proof.
\end{proof}

\begin{lemma}\label{c2pr11}
For $2^j\leq m<2^{j+1}$ and $j\geq 1$, 
(a). the largest integer $q$ such that $\bar w_q(\mathbb{R}P^m)\neq 0$ is $q=2^{j+1}-m-1$;
(b). the largest integer $\lambda$ such that $\bar w_\lambda(\xi_{\mathbb{R}P^m,2})\neq 0$ is $\lambda=2^{j+1}-1$.
\end{lemma}

\begin{proof}
We notice that 
\begin{eqnarray}\label{0930-r1}
H^*(\mathbb{R}P^m;\mathbb{Z}_2)=\mathbb{Z}_2[\alpha]/(\alpha^{m+1})
\end{eqnarray}
where $\alpha$ is a generator of degree $1$. Hence by applying (\ref{0930-r1}) to \cite[Corollary 11.15]{cc}, 
\begin{eqnarray}\label{0930-r2}
w(\mathbb{R}P^m)=(1+\alpha)^{m+1}.
\end{eqnarray}
It follows from (\ref{0930-r2}) that 
\begin{eqnarray}\label{c2eqwr}
\bar w(\mathbb{R}P^m)&=&(1+\alpha)^{-(m+1)}\nonumber\\
&=&1+\sum_{i=1}^{\infty} {{-(m+1)}\choose {i}}\alpha^i\nonumber\\
&=&1+\sum_{i=1}^m(-1)^i{{m+i}\choose {m}}\alpha^i.
\end{eqnarray}
On the other hand, let the dyadic expansions of $m+i$ and $m$ be
\begin{eqnarray*}
m+i&=&(a_t,a_{t-1},\cdots,a_1,a_0)_2,\\
m&=&(b_t,b_{t-1},\cdots,b_1,b_0)_2
\end{eqnarray*}
such that $a_t=1$ and $a_i,b_i\in \{0,1\}$ for all $i=0,1,\cdots,t$. 
By the Lucas Theorem (cf. \cite{lucas}), ${{m+i}\choose {m}}$ is odd if and only if  for each $i=0,1,\cdots,t$, $a_i\geq b_i$.  Hence for $2^j\leq m<2^j$, the largest integer $i$  such that ${{m+i}\choose {m}}$ is odd is $i=2^{j+1}-m-1$. Consequently, it follows from (\ref{c2eqwr}) that the largest integer $q$ such that $\bar w_q (\mathbb{R}P^m)\neq 0$ is  $q=2^{j+1}-m-1$. Hence we obtain (a). 
By applying Corollary~\ref{c2co3}, we obtain (b).
\end{proof}

The following two lemmas could be proved analogously with Lemma~\ref{c2pr11}.  

\begin{lemma}\label{c2pr12}
For $2^j\leq m<2^{j+1}$ and $j\geq 1$, 
(a). the largest integer $q$ such that $\bar w_q(\mathbb{C}P^m)\neq 0$ is $q=2^{j+2}-2m-2$;
(b). the largest integer $\lambda$ such that $\bar w_\lambda(\xi_{\mathbb{C}P^m,2})\neq 0$ is $\lambda=2^{j+2}-2$.
\end{lemma}

\begin{proof}
We note that $$H^*(\mathbb{C}P^m;\mathbb{Z}_2)=\mathbb{Z}_2[\beta]/(\beta^{m+1})$$
where $\beta$ is a generator of degree $2$.  Applying the same argument as in the proof of Lemma~\ref{c2pr11}, we see that the largest $i$ such that $\bar w_i (\mathbb{C}P^m)\neq 0$ is  $i=2^{j+2}-2m-2$.  With the help of Corollary~\ref{c2co3}, we finish the proof.
\end{proof}

\begin{lemma}\label{c2pr20}
For $2^j\leq m<2^{j+1}$ and $j\geq 1$, 
(a). the largest integer $q$ such that $\bar w_q(\mathbb{H}P^m)\neq 0$ is $q=2^{j+3}-4m-4$;
(b). the largest integer $\lambda$ such that $\bar w_\lambda(\xi_{\mathbb{H}P^m,2})\neq 0$ is $\lambda=2^{j+3}-4$.
\end{lemma}

\begin{proof}
We note that $$H^*(\mathbb{H}P^m;\mathbb{Z}_2)=\mathbb{Z}_2[\delta]/(\delta^{m+1})$$
where $\delta$ is a generator of degree $4$.  Applying the same argument as in the proof of Lemma~\ref{c2pr11}, we see that the largest $i$ such that $\bar w_i (\mathbb{H}P^m)\neq 0$ is  $i=2^{j+3}-4m-4$.  With the help of Corollary~\ref{c2co3}, we finish the proof.
\end{proof}

\begin{lemma}\label{1007}
Let $m\geq 4$ and $\kappa$ be the
 largest integer  such that $\bar c_{\kappa}(\xi_{\mathbb{C}P^m,2}^\mathbb{C})\neq 0$.  Then
\begin{eqnarray}\label{27}
\kappa\geq 2m-2. 
\end{eqnarray}
\end{lemma}

\begin{proof}
By Theorem~\ref{thm0503} and (\ref{0930-6}) we obtain 
\begin{eqnarray}\label{20}
c_1(\zeta_{\mathbb{C}P^m,2}^\mathbb{C})=c_1(\xi_{\mathbb{C}P^m,2}^\mathbb{C})=au+bc_1
\end{eqnarray}
with $a,b\in\mathbb{Z}$ and $u^2=c_1u$.  
And from (\ref{0703c}) and (\ref{20}) we obtain
\begin{eqnarray}\label{21}
\bar c_t(\xi_{\mathbb{C}P^m,2}^{\mathbb{C}})&=&(-1)^t(au+bc_1)^t\nonumber\\
&=&(-1)^t(\sum_{i=1}^ta^ib^{t-i})c_1^{t-1}u+(-1)^tb^tc_1^t.
\end{eqnarray}

On the other hand,  let $(\xi_{\mathbb{C}P^m,2}^\mathbb{C})_\mathbb{R}$ denote the underlying real vector bundle of $\xi_{\mathbb{C}P^m,2}^\mathbb{C}$. Then
\begin{eqnarray}\label{0703-51}
(\xi_{\mathbb{C}P^m,2}^\mathbb{C})_\mathbb{R}\cong(\xi_{\mathbb{C}P^m,2})^{\oplus 2}.
\end{eqnarray}
Let 
%\begin{eqnarray}\label{17}
$v=w_1(\xi_{\mathbb{C}P^m,2})$.
%\end{eqnarray}
It follows from (\ref{0703r}) and (\ref{0703-51})  that %and (\ref{17}) that
\begin{eqnarray}\label{23}
w((\xi_{\mathbb{C}P^m,2}^\mathbb{C})_\mathbb{R})&=&(w(\xi_{\mathbb{C}P^m,2}))^2\nonumber\\
&=&(1+v)^2\nonumber\\
&=&1+v^2.
\end{eqnarray}
By Theorem~\ref{thm0419}, $v^2\neq 0$. Hence from (\ref{23}), $w((\xi_{\mathbb{C}P^m,2}^\mathbb{C})_\mathbb{R})\neq 1$. Hence $(\xi_{\mathbb{C}P^m,2}^\mathbb{C})_\mathbb{R}$ is not  trivial. Consequently, as a complex vector bundle, $\xi_{\mathbb{C}P^m,2}^\mathbb{C}$ is not trivial. It follows that  $\zeta_{\mathbb{C}P^m,2}^\mathbb{C}$ is not trivial as a complex line bundle. 
Therefore,  we see that  the first Chern class of $\zeta_{\mathbb{C}P^m,2}^\mathbb{C}$ is nonzero, i.e.   in (\ref{20}), either $a\neq 0$ or $b\neq 0$.

\noindent{\sc{Case~1}}: $b\neq 0$. 

\noindent
Then by (\ref{21}), we have $\bar c_t(\xi_{\mathbb{C}P^m,2}^{\mathbb{C}})\neq 0$ whenever $c_1^t\neq 0$. Noting that $c_1^t=0$ if and only if $t\geq h(m)+1$, we conclude that
\begin{eqnarray}\label{24}
\kappa\geq h(m). 
\end{eqnarray}

\noindent{\sc{Case~2}}: $b=0$.

\noindent Then $a\neq 0$ and by (\ref{21}),
\begin{eqnarray}\label{26}
\bar c_t(\xi_{\mathbb{C}P^m,2}^{\mathbb{C}})&=&(-1)^ta^tu^t\nonumber\\
&=&(-1)^ta^tc_1^{t-1}u.
\end{eqnarray}
We note that $c_1^{t-1}u= 0$ if and only if $t-1\geq h(m)+1$. Hence by (\ref{26}), 
\begin{eqnarray}\label{25}
\kappa=h(m)+1. 
\end{eqnarray}
Summarizing both {\sc{Case~1}} and {\sc{Case~2}}, it follows from (\ref{24}) and (\ref{25}) that $\kappa(m)\geq h(m) $. With the help of Proposition~\ref{schubert2} we obtain (\ref{27}).
\end{proof}

\section{Proof of Theorem~\ref{0921-3}}\label{sec5}

\begin{proof}[Proof of Theorem~\ref{0921-3}]  
Firstly, we have
\begin{eqnarray*}
& &w(\prod_{i=1}^{k_1}S^{m_{1,i}}\times \prod_{j=1}^{k_2}\mathbb{R}P^{m_{2,j}} \times \prod_{t=1}^{k_3}\mathbb{C}P^{m_{3,t}}\times \prod_{l=1}^{k_4}\mathbb{H}P^{m_{4,l}})\\
&=&\prod_{i=1}^{k_1}w(S^{m_{1,i}})\times \prod_{j=1}^{k_2}w(\mathbb{R}P^{m_{2,j}} )\times \prod_{t=1}^{k_3}w(\mathbb{C}P^{m_{3,t}})\times \prod_{l=1}^{k_4}w(\mathbb{H}P^{m_{4,l}}).
\end{eqnarray*}
It follows that 
\begin{eqnarray}\label{eq0921-8}
& &\bar w(\prod_{i=1}^{k_1}S^{m_{1,i}}\times \prod_{j=1}^{k_2}\mathbb{R}P^{m_{2,j}} \times \prod_{t=1}^{k_3}\mathbb{C}P^{m_{3,t}}\times \prod_{l=1}^{k_4}\mathbb{H}P^{m_{4,l}})\nonumber\\
&=& \prod_{i=1}^{k_1}\bar w(S^{m_{1,i}})\times \prod_{j=1}^{k_2}\bar w(\mathbb{R}P^{m_{2,j}} )\times\prod_{t=1}^{k_3}\bar w(\mathbb{C}P^{m_{3,t}})\nonumber\\
&&  \times \prod_{l=1}^{k_4}\bar w(\mathbb{H}P^{m_{4,l}}).
\end{eqnarray}  

Secondly, it follows from (\ref{eq0921-8}) and Lemma~\ref{c2pr9} - Lemma~\ref{c2pr20} that
the largest integer $q$ such that 
$$
\bar w_q(\prod_{i=1}^{k_1}S^{m_{1,i}}\times \prod_{j=1}^{k_2}\mathbb{R}P^{m_{2,j}} \times \prod_{t=1}^{k_3}\mathbb{C}P^{m_{3,t}}\times \prod_{l=1}^{k_4}\mathbb{H}P^{m_{4,l}})\neq 0
$$
is 
\begin{eqnarray}
q&=&\sum_{j=1}^{k_2}(2^{[\log_2 m_{2,j}]+1}-m_{2,j}-1)+\sum_{t=1}^{k_3}(2^{[\log_2 m_{3,t}]+2}-2m_{3,t}-2 )+
\nonumber\\&&\sum_{l=1}^{k_4}(2^{[\log_2 m_{4,l}]+3}-4m_{4,l}-4 )\nonumber\\
&=&\sum_{j=1}^{k_2}(2^{[\log_2 m_{2,j}]+1}-m_{2,j})+\sum_{t=1}^{k_3}(2^{[\log_2 m_{3,t}]+2}-2m_{3,t} )\nonumber\\&&+\sum_{l=1}^{k_4}(2^{[\log_2 m_{4,l}]+3}-4m_{4,l})
-k_2-2k_3-4k_4.\label{0930-1}
\end{eqnarray}

Since $\prod_{i=1}^{k_1}S^{m_{1,i}}\times \prod_{j=1}^{k_2}\mathbb{R}P^{m_{2,j}} \times \prod_{t=1}^{k_3}\mathbb{C}P^{m_{3,t}}\times \prod_{l=1}^{k_4}\mathbb{H}P^{m_{4,l}}$ is a closed and connected manifold, by (\ref{0930-1}) and Corollary~\ref{c2co3}, the largest integer $\lambda$ such that 
$$
\bar w_\lambda(\xi_{\prod_{i=1}^{k_1}S^{m_{1,i}}\times \prod_{j=1}^{k_2}\mathbb{R}P^{m_{2,j}} \times \prod_{t=1}^{k_3}\mathbb{C}P^{m_{3,t}}\times \prod_{l=1}^{k_4}\mathbb{H}P^{m_{4,l}},2})\neq 0
$$
is 
\begin{eqnarray}\label{0930-2}
\lambda&=&q+\dim (\prod_{i=1}^{k_1}S^{m_{1,i}}\times \prod_{j=1}^{k_2}\mathbb{R}P^{m_{2,j}} \times \prod_{t=1}^{k_3}\mathbb{C}P^{m_{3,t}}\times \prod_{l=1}^{k_4}\mathbb{H}P^{m_{4,l}})\nonumber\\
&=&\sum_{i=1}^{k_1} m_{1,i}+\sum_{j=1}^{k_2}2^{[\log_2 m_{2,j}]+1}+\sum_{t=1}^{k_3}2^{[\log_2 m_{3,t}]+2}\nonumber\\&& +\sum_{l=1}^{k_4}2^{[\log_2 m_{4,l}]+3}-k_2-2k_3-4k_4.
\end{eqnarray}

Finally, from Lemma~\ref{cor1} and (\ref{0930-2}), we obtain Theorem~\ref{0921-3}. 
\end{proof}

Corollary~\ref{rpm} is obtained from Theorem~\ref{0921-3} immediately. 
Corollary~\ref{sphere} is obtained from  \cite[Example~2.6-(2)]{high1}  and Theorem~\ref{0921-3}.  And Corollary~\ref{exactrpm} is obtained from  Example~\ref{upperrpm} and Theorem~\ref{0921-3}.

\section{An obstruction for regular maps on disjoint unions}\label{6.2}

The main purpose of this section is to give an obstruction for the regular maps on disjoint unions of manifolds defined in Definition~\ref{c2def1}.  As a by-product, we re-obtain \cite[Theorem 2.4]{handel2} in Remark~\ref{1007-r2}. %The obstruction given in this section (Proposition~\ref{c2co1} and Proposition~\ref{c2co114}) is a generalization of Lemma~\ref{cor1} and Lemma~\ref{cor2}. 
Throughout this section, all Chern classes are with mod $p$ coefficients for some primes $p$.  

\smallskip

We consider the following canonical projections
\begin{eqnarray*}
\pi_i: \prod _{i=1}^n(F(M_i,k_i)/\Sigma_i)\longrightarrow F(M_i,k_i)/\Sigma_i. 
\end{eqnarray*}
By applying the Kunneth formula, we see that for each $i=1,2,\cdots,n$, the map $\pi_i$ induces a monomorphism on cohomology
$$
\xymatrix{
\pi_i^*: H^*(F(M_i,k_i)/\Sigma_i;{\bf{k}})\ar[r]&H^*(\prod _{i=1}^n(F(M_i,k_i)/\Sigma_i);{\bf{k}})\ar@{=}[d]\\
&\bigotimes_{i=1}^n H^*(F(M_i,k_i)/\Sigma_i;{\bf{k}})
}
$$
where ${\bf{k}}$ is an arbitrary field.  The tensor product of these maps is an isomorphism
\begin{eqnarray}\label{0921-5}
\bigotimes_{i=1}^n\pi_i^*: \bigotimes _{i=1}^nH^*(F(M_i,k_i)/\Sigma_i;{\bf{k}})\overset{\cong}\longrightarrow \bigotimes_{i=1}^n H^*(F(M_i,k_i)/\Sigma_i;{\bf{k}}).
\end{eqnarray}

\smallskip

Let $\xi_{M_1,k_1;\cdots;M_n,k_n}^{\mathbb{F}}$ denote the following vector bundle
$$
\xymatrix{
&\mathbb{F}^{\sum_{i=1}^n k_i}\cong \prod_{i=1}^n\mathbb{F}^{k_i}\ar[r]
&\prod_{i=1}^n(F(M_i,k_i)\times_{\Sigma_{k_i}}\mathbb{F}^{k_i} )\ar[d]\\
&&\prod_{i=1}^n(F(M_i,k_i)/\Sigma_{k_i} ).}
$$
For simplicity, we write $\xi_{M_1,k_1;\cdots;M_n,k_n}^{\mathbb{R}}$ as $\xi_{M_1,k_1;\cdots;M_n,k_n}^{}$.  For each $i=1,2,\cdots,n$,  it can be verified immediately that $\pi_i$ induces a pull-back of vector bundles by the following commutative diagram 
\begin{eqnarray}\label{6.5d1}
\xymatrix{
\mathbb{F}^{k_i}\ar[d]\ar@{=}[r]& \mathbb{F}^{k_i}\ar[d]\\
\prod_{ t\neq i}(F(X_t,k_t)/\Sigma_{k_t} )\times (F(M_i,k_i)\times_{\Sigma_{k_i}}\mathbb{F}^{k_i} )\ar[d]\ar[r]& (F(M_i,k_i)\times_{\Sigma_{k_i}}\mathbb{F}^{k_i} )\ar[d]\\
\prod_{t=1}^n(F(X_t,k_t)/\Sigma_{k_t} )\ar[r]^{\pi_i} & F(M_i,k_i)/\Sigma_{k_i}.
}
\end{eqnarray}
Moreover, taking the Whitney sum of  the vector bundles given in the left column of (\ref{6.5d1}) with $i$ going from $1$ to $n$, we can recover the vector bundle $\xi_{M_1,k_1;\cdots;M_n,k_n}^{\mathbb{F}}$. 
\begin{proposition}\label{c2pr4}
We have the isomorphism of vector bundles
\begin{eqnarray*} 
\xi_{M_1,k_1;\cdots;M_n,k_n}^{\mathbb{F}}\cong \bigoplus _{i=1}^n \pi_i^*\xi_{M_i, k_i}^\mathbb{F}.
\end{eqnarray*}
\end{proposition}
The following corollary follows from Proposition~\ref{c2pr4}.
\begin{corollary}\label{c2pr99}
(a). Under the identification given by the isomorphism  (\ref{0921-5}) with $\mathbb{Z}_2$-coefficient, the Stiefel-Whitney classes saitsfy
\begin{eqnarray*} 
w(\pi_i^*\xi_{M_i,k_i})=\pi_i^*w(\xi_{M_i,k_i})=w(\xi_{M_i,k_i}).
\end{eqnarray*}

(b). Under the identification given by the isomorphism  (\ref{0921-5}) with $\mathbb{Z}_p$-coefficient for any primes $p$, the Chern classes saitsfy
\begin{eqnarray*} 
c(\pi_i^*\xi_{M_i,k_i}^{\mathbb{C}})=\pi_i^*c(\xi_{M_i,k_i}^\mathbb{C})=c(\xi_{M_i,k_i}^\mathbb{C}).
\end{eqnarray*}
\end{corollary}

The next propositions follows by a straight-forward generalization of   \cite[Lemma~2.11]{high1} or   \cite[Proposition~2.1]{cohen1}. %  by investigating the  inverse bundle  of $\xi_{M_1,k_1;\cdots; M_n,k_n}$.   

\begin{proposition}\label{c2pr8}
Suppose there is a $(M_1,k_1;M_2,k_2;\cdots;M_n,k_n)$-regular map $f:\coprod_{i=1}^n M_i\longrightarrow \mathbb{R}^N$. If 
\begin{eqnarray*}
\bar w_r(\xi_{M_1,k_1;\cdots; M_n,k_n})\neq 0,
\end{eqnarray*}
then $N\geq r+\sum_{i=1}^n k_i$. 
\end{proposition}
%\begin{proof}
%It follows from (\ref{0930-9}) that either $t\leq N-\sum_{i=1}^n k_i $ or $t\geq N+1$. Suppose $t\geq N+1$. For any $m\geq N$, there exists a $(M_1,k_1;M_2,k_2;\cdots;M_n,k_n)$-regular map $i\circ f$ from $\coprod_{i=1}^n M_i$ into $\mathbb{R}^m$ where $i$ is the inclusion of $\mathbb{R}^N$ into $\mathbb{R}^m$. Hence either $t\leq m-k $ or $t\geq m+1$. Letting $m=N+1,N+2,\cdots$ subsequently, we obtain $t\geq N+2$, $t\geq N+3$, $\cdots$. This contradicts that $t$ is finite. Therefore, $t\leq N-k$. 
%\end{proof}

%\begin{remark}

%\end{remark}

The next proposition follows from Proposition~\ref{c2pr8}. 

\begin{proposition}\label{c2co1}
Let $\lambda_i$ be the largest integer such that 
\begin{eqnarray*}
\bar w_{\lambda_i}(\xi_{M_i,k_i})\neq 0.
\end{eqnarray*}
If there is a $(M_1,k_1;M_2,k_2;\cdots;M_n,k_n)$-regular map $f:\coprod_{i=1}^n M_i\longrightarrow \mathbb{R}^N$, then 
\begin{eqnarray*}\label{c2eq6}
N\geq \sum_{i=1}^n (\lambda_i+k_i). 
\end{eqnarray*}
\end{proposition}
\begin{proof}
It follows from Proposition~\ref{c2pr4} and Corollary~\ref{c2pr99} (a)  that 
\begin{eqnarray}\label{c2eq3}
w(\xi_{M_1,k_1;M_2,k_2;\cdots;M_n,k_n})&=&w(\bigoplus_{i=1}^n \pi^*_i \xi_{M_i,k_i})\nonumber\\
&=& \prod_{i=1}^n w (\pi^*_i \xi _{M_i,k_i})\nonumber\\
&=&\prod_{i=1}^n w(\xi_{M_i,k_i})\nonumber.
\end{eqnarray}
Therefore, 
\begin{eqnarray}\label{c2eq4}
\bar w(\xi_{M_1,k_1;M_2,k_2;\cdots;M_n,k_n})=\prod_{i=1}^n \bar w(\xi_{M_i,k_i}).
\end{eqnarray}
By the definition of $\lambda_i$'s, it follows from (\ref{c2eq4}) that
\begin{eqnarray}\label{c2eq5}
\bar w_{\sum_{i=1}^n\lambda_i}(\xi_{M_1,k_1;M_2,k_2;\cdots;M_n,k_n})=\prod_{i=1}^n \bar w_{\lambda_i}(\xi_{M_i,k_i})\neq 0.
\end{eqnarray}
Suppose there is a $(M_1,k_1;M_2,k_2;\cdots;M_n,k_n)$-regular map $f:\coprod_{i=1}^n M_i\longrightarrow \mathbb{R}^N$. Then by applying (\ref{c2eq5}) to Proposition~\ref{c2pr8}, we obtain Proposition~\ref{c2co1}. 
\end{proof}

\begin{corollary}\label{c2co2}
For $i=1,2,\cdots,n$, let $M_i$ be a closed, connected $m_i$-dimensional manifold and $h_i$ be the largest integer such that 
\begin{eqnarray*}
 \bar w_{h_i}(M_i)\neq 0.
\end{eqnarray*}
If there is a $(M_1,2;M_2,2;\cdots;M_n,2)$-regular map of $\coprod_{i=1}^n M_i$ into $\mathbb{R}^N$, then 
\begin{eqnarray*} 
N\geq \sum_{i=1}^n (m_i+h_i)+2n. 
\end{eqnarray*}
\end{corollary}
\begin{proof}
By Corollary~\ref{c2co3}, we see that the largest integer $\lambda_i$ such that 
\begin{eqnarray*}
 \bar w_{\lambda_i}(\xi_{M_i,2})\neq 0
\end{eqnarray*}
is $\lambda_i=m_i+h_i$. By Proposition~\ref{c2co1}, we finish the proof.
\end{proof}
\begin{remark}\label{1007-r2}
Since any $2n$-regular map on $\coprod_{i=1}^n M_i$ is a $(M_1,2;$ $M_2,2;$ $\cdots;$ $M_n,2)$-regular map, we re-obtain \cite[Theorem 2.4]{handel2}  from  Corollary~\ref{c2co2}. 
\end{remark}

%{{\sc{Case~2}}: $\mathbb{F}=\mathbb{C}$.}

%Let $p$ be an arbitrary prime. From Proposition~\ref{c2pr6}, we have an induced commutative diagram on cohomology as follows
%$$
%\xymatrix
%{
%\bigotimes_{i=1}^nH^*(F(M_i,k_i)/\Sigma_i;\mathbb{Z}_p)
%&\ar[l]_>>>>>{h^*}\dfrac{\mathbb{Z}_p[c_1,\cdots,c_{\sum_{i=1}^nk_i}]}{(\bar c_{N-\sum_{i=1}^nk_i+1},\cdots, \bar c_N)}\\
%&\mathbb{Z}_p[c_1,\cdots,c_{\sum_{i=1}^nk_i}]\ar[u]^>>>>>{\iota^*}\ar[lu]^{\phi^*}. 
%}
%$$
%It follows that for $N-\sum_{i=1}^n k_i+1\leq t \leq N$, 
%\begin{eqnarray}
%\bar c_{t}(\xi_{M_1,k_1;\cdots; M_n,k_n}^\mathbb{C})&=&\phi^*\bar c_{t}\nonumber\\
%&=& h^*i^* \bar c_{t}\nonumber\\
%&=&0. \label{0930-8}
%\end{eqnarray}
%The following proposition follows from (\ref{0930-8}) consequently.  With minor modifications,  the proof of Proposition~\ref{c2pr8} is applicable here. 

The next proposition is a straight-forward generalization of \cite[Lemma~5.7]{high2}.
 
\begin{proposition}\label{c2pr113}
Suppose there is a complex $(M_1,k_1;$ $M_2,k_2;$ $\cdots;$ $M_n,k_n)$-regular map $f:\coprod_{i=1}^n M_i\longrightarrow \mathbb{C}^N$.  If 
\begin{eqnarray*}
\bar c_r(\xi_{M_1,k_1;\cdots; M_n,k_n}^\mathbb{C})\neq 0,
\end{eqnarray*}
then $N\geq r+\sum_{i=1}^n k_i$. 
\end{proposition}

%\begin{remark}
%Generalizing the proof of \cite[Lemma~5.7]{high2}, Proposition~\ref{c2pr113} can be obtained in an alternative way by investigating the complex inverse bundle of $\xi_{M_1,k_1;\cdots; M_n,k_n}^\mathbb{C}$.   
%\end{remark}

The next proposition follows from Proposition~\ref{c2pr113}. 

\begin{proposition}\label{c2co114}
Let $\tau_i$ be the largest integer such that 
\begin{eqnarray*}
\bar c_{\tau_i}(\xi_{M_i,k_i}^\mathbb{C})\neq 0.
\end{eqnarray*}
If there is a complex $(M_1,k_1;M_2,k_2;\cdots;M_n,k_n)$-regular map $f:\coprod_{i=1}^n M_i\longrightarrow \mathbb{C}^N$, then 
\begin{eqnarray*} 
N\geq \sum_{i=1}^n (\tau_i+k_i). 
\end{eqnarray*}
\end{proposition}
\begin{proof}
It follows from Proposition~\ref{c2pr4} and Corollary~\ref{c2pr99} (b)  that 
\begin{eqnarray}\label{c2eq115}
c(\xi_{M_1,k_1;M_2,k_2;\cdots;M_n,k_n}^\mathbb{C})&=&c(\bigoplus_{i=1}^n \pi^*_i \xi_{M_i,k_i}^\mathbb{C})\nonumber\\
&=& \prod_{i=1}^n c (\pi^*_i \xi _{M_i,k_i}^\mathbb{C})\nonumber\\
&=&\prod_{i=1}^n c(\xi_{M_i,k_i}^\mathbb{C})\nonumber.
\end{eqnarray}
Therefore, 
\begin{eqnarray}\label{c2eq114}
\bar c(\xi_{M_1,k_1;M_2,k_2;\cdots;M_n,k_n}^\mathbb{C})=\prod_{i=1}^n \bar c(\xi_{M_i,k_i}^\mathbb{C}).
\end{eqnarray}
By the definition of $\tau_i$'s, it follows from (\ref{c2eq114}) that
\begin{eqnarray}\label{c2eq119}
\bar c_{\sum_{i=1}^n\tau_i}(\xi_{M_1,k_1;M_2,k_2;\cdots;M_n,k_n}^\mathbb{C})=\prod_{i=1}^n \bar c_{\tau_i}(\xi_{M_i,k_i}^\mathbb{C})\neq 0.
\end{eqnarray}
Suppose there is a complex $(M_1,k_1;M_2,k_2;\cdots;M_n,k_n)$-regular map $f:\coprod_{i=1}^n M_i\longrightarrow \mathbb{C}^N$. Then by applying (\ref{c2eq119}) to Proposition~\ref{c2pr113}, we obtain Proposition~\ref{c2co114}. 
\end{proof}

\section{Proof of Theorem~\ref{c2th1}}\label{6.3}

\begin{proof}[Proof of Theorem~\ref{c2th1}]
By Lemma~\ref{c2pr10}, the largest integer $\lambda_{0,s}$ such that 
$
\bar w_{\lambda_{0,s}}(\xi_{\mathbb{R}^2,2^{d_s}})\neq 0
$ 
  is  
  $
 \lambda_{0,s}=2^{d_s}-1.
 $
  By Lemma~\ref{c2pr9}, the largest integer $\lambda_{1,i}$ such that 
 $
 \bar w_{\lambda_{1,i}}(\xi_{S^{m_{1,i}},2})\neq 0
 $ 
 is 
  $
 \lambda_{1,i}=m_{1,i}.
 $
  By Lemma~\ref{c2pr11}, the largest integer $\lambda_{2,j}$ such that 
 $
 \bar w_{\lambda_{2,j}}(\xi_{\mathbb{R}P^{m_{2,j}},2})\neq 0
 $ 
 is  
  $\lambda_{2,j}=2^{[\log_2m_{2,j}]+1}-1.
 $
  By Lemma~\ref{c2pr12}, the largest integer $\lambda_{3,t}$ such that 
 $
 \bar w_{\lambda_{3,t}}(\xi_{\mathbb{C}P^{m_{3,t}},2})\neq 0
 $ 
 is 
  $\lambda_{3,t}=2^{[\log_2m_{3,t}]+2}-2.
 $   
 By Lemma~\ref{c2pr20}, the largest integer $\lambda_{4,l}$ such that 
$
\bar w_{\lambda_{4,l}}(\xi_{\mathbb{H}P^{m_{4,l}},2})\neq 0
$ 
is 
 $\lambda_{4,l}=2^{[\log_2m_{4,l}]+3}-4.
$   
By applying   
all the above to Proposition~\ref{c2co1}, we have
\begin{eqnarray*}
N&\geq& \sum_{s=1}^{k_0}(\lambda_{0,s}+2^{d_s})+\sum_{i=1}^{k_1}(\lambda_{1,i}+2)+\sum_{j=1}^{k_2}(\lambda_{2,j}+2)\\
& &+\sum_{t=1}^{k_3}(\lambda_{3,t}+2)+\sum_{l=1}^{k_4}(\lambda_{4,l}+2)\\
&=& \sum_{s=1}^{k_0}2^{d_s+1}+\sum_{i=1}^{k_1}m_{1,i}+\sum_{j=1}^{k_2}2^{[\log_2m_{2,j}]+1}\\
& &+\sum_{t=1}^{k_3}2^{[\log_2m_{3,t}]+2}+\sum_{l=1}^{k_4}2^{[\log_2m_{4,l}]+3}-k_0+2k_1+k_2-2k_4.
\end{eqnarray*}
The assertion follows.
\end{proof}

Corollary~\ref{c2th3} follows immediately from Theorem~\ref{c2th1} and  Example~\ref{c2pr2}. 

\section{Complex regular maps on Euclidean spaces, spheres and complex projective spaces}\label{sec8}

\begin{proposition}\label{c2th888}
For    an odd prime $p$, 
if there exists a complex $(\mathbb{R}^{m_1},p;$ $\cdots;$ $\mathbb{R}^{m_{k}},p)$-regular map
\begin{eqnarray*}
f: \coprod_{j=1}^{k } \mathbb{R}^{m_j}\longrightarrow \mathbb{C}^N,
\end{eqnarray*}
then
\begin{eqnarray*}
N\geq \sum_{j=1}^{k } [\frac{m_j+1}{2}](p-1)+ k .
\end{eqnarray*}
\end{proposition}

Suppose there exists a   complex $np$-regular map $f: \mathbb{R}^m\longrightarrow \mathbb{C}^N$. 
Then the composition 
\begin{eqnarray*}
\coprod_{i=1}^n \mathbb{R}^m\overset{\cong}{\longrightarrow} \coprod_{i=1}^n {D}^m\lhook\joinrel\relbar\joinrel\rightarrow \mathbb{C}^m\overset{f}{\longrightarrow}\mathbb{C}^N
\end{eqnarray*}
gives a complex $(\underbrace{\mathbb{R}^m,p;\cdots;\mathbb{R}^m,p}_{n})$-regular map from $\coprod_{i=1}^n \mathbb{R}^m$ into $\mathbb{C}^N$ where $D^m$ denotes the unit open ball in $\mathbb{R}^m$.   Hence as a consequence of Proposition~\ref{c2th888},  we  obtain the next proposition.  
\begin{proposition}\label{th1025-1}
Let $p$ be an odd prime.  If there exists a complex $np$-regular map from $ \mathbb{R}^m$ into  $\mathbb{C}^N$, then $N\geq n([\frac{m+1}{2}] (p-1)+1)$. 
\end{proposition}

\begin{proof}[Proof of Proposition~\ref{c2th888}]
Throughout the proof, we let the cohomology coefficients to be $\mathbb{Z}_p$, where $p$ is the given odd prime.  Let $\nu_j$ be the largest integer such that the  Chern class with mod $p$ coefficients 
\begin{eqnarray*}
\bar c_{\nu_j}(\xi_{\mathbb{R}^{m_j},p}^\mathbb{C})\neq 0. 
\end{eqnarray*}
 We notice that $\nu_j$ must be finite. On the other hand, in \cite[Proof of Theorem~5.2]{high2}, it is proved that the Chern class with mod $p$ coefficients satisfies 
\begin{eqnarray*}
\bar c_{[\frac{m_j-1}{2}](p-1)}(\xi^\mathbb{C}_{\mathbb{R}^{m_j},p})\neq 0.
\end{eqnarray*}
 Hence 
 \begin{eqnarray*}
 \nu_j\geq [\frac{m_j-1}{2}](p-1).
 \end{eqnarray*}
  With the help of Proposition~\ref{c2co114}, we have
\begin{eqnarray*}
N&\geq& \sum_{j=1}^{k }(\nu_j+p)\\
&\geq & 
\sum_{j=1}^{k } [\frac{m_j+1}{2}](p-1)+ k.
\end{eqnarray*}
The assertion follows. 
\end{proof}

The next proposition follows   from   Lemma~\ref{cor2} and Lemma~\ref{0922-le3}.

\begin{proposition}
\label{complex}
If there exists a  complex $2$-regular map of $S^m$ into $\mathbb{C}^N$, then $N\geq [\frac{m}{2}] +2$.  
\end{proposition}
 
The next proposition follows  from   Lemma~\ref{cor2} and Lemma~\ref{1007}.

\begin{proposition}\label{cpm2}
Let $m\geq 4$. If there exists a complex $2$-regular map of $\mathbb{C}P^m$ into $\mathbb{C}^N$, then $N\geq 2m$. 
\end{proposition}

\medskip
{ 
\noindent{\bf{Acknowledgement}}. The  author would like to express his deep gratitude to Prof. Jie Wu for his guidance and help.% and to the referee for the careful reading of the manuscript and valuable comments to improve the paper.
}

\vspace{1.5cm}

Department of Mathematics, 

National University of Singapore

Email Address: sren@u.nus.edu

\end{document}